\documentclass{birkjour}
 \newtheorem{thm}{Theorem}[section]
 \newtheorem{cor}[thm]{Corollary}
 \newtheorem{lem}[thm]{Lemma}
 \newtheorem{prop}[thm]{Proposition}
 \theoremstyle{definition}
 
 \theoremstyle{remark}

 \numberwithin{equation}{section}

\begin{document}

%-------------------------------------------------------------------------
% editorial commands: to be inserted by the editorial office
%
%\firstpage{1} \volume{228} \Copyrightyear{2004} \DOI{003-0001}
%\seriesextraline{This is the Concrete Title of this Book\br H.E. R and S.T.C. W, Eds.}
%\firstpage{1}

\title[Harnack Estimate For Positive Solutions to a Nonlinear Equation]
 {Harnack Estimate For Positive Solutions to a Nonlinear Equation Under Geometric Flow}

%----------Author 1
\author[Gh. Fasihi]{Gh. Fasihi-Ramandi}

\address{%
Imam Khomeini International University\\
Noroozian Avenue\\
Qazvin\\
Iran}

\email{fasihi@sci.ikiu.ac.ir}

%\thanks{This work was completed with the support of our
%----------Author 2
\author{Sh. Azami}
\address{Imam Khomeini International University\br
Qazvin\br
Iran}
\email{azami@sci.ikiu.ac.ir}


%\date{January 1, 2004}

\begin{abstract}
In the present paper, we obtain some gradient estimates for positive solutions to the following nonlinear parabolic equation under general geometric flow on complete noncompact manifolds
\[\dfrac{\partial u}{\partial t}=\triangle u +a(x,t)u^p +b(x,t)u^q\]
where, $0<p,q<1$ are real constants and $a(x,t)$ and $b(x,t)$ are functions which are $C^2$ in the $x$-variable and $C^1$ in the $t$-variable. We shall get an interesting Harnack inequality as an application.
\end{abstract}

%%% ----------------------------------------------------------------------
\maketitle
%%% ----------------------------------------------------------------------
%\tableofcontents
\section{Introduction and Main Results}
Starting with the pioneering work of P. Li and S. T. Yau in the paper \cite{Li}, gradient estimates are also called differential Harnack inequalities, because one can obtain the classical Harnack inequality after integrating the gradient estimate along paths in space-time. These concepts are very powerful tools in geometric analysis. For example, R. Hamilton established differential Harnack inequalities for the mean curvature along the mean curvature flow and for  the scalar curvature along the Ricci flow. Both have important applications in the analysis of singularities. \\

In Perelman's work on the Poincare's conjecture and the geometrization conjecture, differential Harnack inequality played an important role. Since then, there have been many works on gradient estimates along the Ricci flow or the conjugate Ricci flow for the solution of the heat equation or the conjugate heat equation; examples include (\cite{4}, \cite{6}). Later, Sun \cite{lem} extended these results to general geometric flow.\\

In the present paper, we study the following nonlinear parabolic equation under general geometric flow on complete noncompact manifolds $M$,
\begin{equation}\label{equ1}
\dfrac{\partial u}{\partial t}=\triangle u +a(x,t)u^p +b(x,t)u^q
\end{equation}
where, $0<p,q<1$ are real constants and $a(x,t)$ and $b(x,t)$ are functions which are $C^2$ in the $x$-variable and $C^1$ in the $t$-variable. Before presenting our main results about the equation, it seems necessary to support our idea of considering this equation. If $a(x,t)$ and $b(x,t)$ are identically zero, the (\ref{equ1}) is the heat equation. As $p\to 1$ and $b(x,t)=0$, the conjugate heat equation
\[\triangle u-Ru+\partial_t u=0,\]
becomes a special case of (\ref{equ1}), where $R$ stands for the scalar curvature. In bio-mathematics, the following equation 
\[\dfrac{\partial u}{\partial t}=\triangle u +a(x,t)u^p, \qquad p>0,\]
could be interpreted as the population dynamics. Also, similar equation arises in the study of conformally deformation of the
scalar curvature on a manifolds (See \cite{wu}, equation (1.4)). 

Now we present our main results about the equation (\ref{equ1}) as follows.
\begin{prop}\label{maintheorem}
Suppose $(M,g(t))$ be the family of Riemannian manifolds mentioned before.  Let $M$ be complete under the initial metric
g(0). Given $x_0\in M$, and $M_1, R > 0$, let $u$ be a positive solution to the nonlinear
equation $(\ref{equ1})$ with $u\leq  M_1$ in the cube $Q_{2R,T} = \{(x, t) | d(x, x_0, t) \leq 2R, 0 \leq
t \leq T \}$. Suppose that there exist constants $K_1, K_2, K_3, K_4 \geq 0$ such that 
\[Ric \geq −K_1 g,\quad -K_2 g  \leq s \leq K_3 g,\quad |\nabla s| \leq K_4\]
 on $Q_{2R,T}$. Moreover, assume that there exist positive constants $\theta_a,\theta_b,\gamma_a,\gamma_b$ such that $\triangle a \leq \theta_a$, $|\nabla a| \leq  \gamma_a$, $\triangle b \leq \theta_b$ and $|\nabla b| \leq  \gamma_b$ in $Q_{2R,T}$. Then we have
$$\beta \dfrac{|\nabla u|^2}{u^2}+au^{p-1}+bu^{q-1}-\dfrac{u_t}{u}
\leq H_1+H_2+\dfrac{n}{\beta}\dfrac{1}{t},$$
where,
\begin{align*}
H_1&=\dfrac{n}{\beta} \Big( \dfrac{(n-1)(1+\sqrt{K_1}R)c_1^2+c_2+2c_1^2}{R^2}+\sqrt{c_3}K_2+|a|(1-p)M_1^{(p-1)}\\&
+|b|(1-q)M_1^{(q-1)}  +\dfrac{n c_1^2 }{2R^2(\beta -\beta^2)}\Big)
\end{align*}
\begin{align*}
H_2 &=\Big[\dfrac{n^2 }{4\beta^2 (1-\beta)^2}\big( 2(1-\beta )K_3 +2\beta K_1 +\dfrac{3}{2}K_4 \big)^2\\
&+\dfrac{n}{\beta}\{M_1^{(p-1)}\theta_a + M_1^{(q-1)}\theta_b
+n \big( \dfrac{1}{\beta}(K_2 +K_3)^2 +\dfrac{3}{2}K_4\big)\}\\
&-\dfrac{n}{\beta}\{ \dfrac{[
 (p-\beta)M_1^{(p-1)}\gamma_a +(q-\beta)M_1^{(q-1)}\gamma_b]^2}{|a|(p-\beta)(p-1)M_1^{(p-1)}+|b|(q-\beta)(q-1)M_1^{(q-1)}}\} \Big]^{1\over 2}
\end{align*}
\end{prop}
When $R$ approaches to infinity, we can get the global Li-Yau type gradient estimates (see \cite{Li}) for equation (\ref{equ1}) as follows.
\begin{cor}\label{corol}
Let $(M, g(0))$ be a complete noncompact Riemannian manifold without boundary, and suppose that $g(t)$ evolve by $\dfrac{\partial g_{ij}}{\partial t}=2s_{ij}$ for $t\in [0, T ]$ and satisfy 
\[
Ric\geq −K_1 g,\quad −K_2 g\leq s \leq K_3g,\quad |\nabla s| \leq K_4.
\]
Also, assume that $\triangle a \leq \theta_a$, $\triangle b \leq \theta_b$, $|\nabla a| \leq \gamma_a$ and $|\nabla b| \leq \gamma_b$ in $M \times [0, T )$ for some constants $\theta_a ,\theta_b , \gamma_a$ and $\gamma_b$. Let $u$ be a positive solution of (\ref{equ1}) with $u \leq M_1$. Then  for any constant $0 < \beta < 1$, if $\beta < p,q < 1$, we have
\[\beta \dfrac{|\nabla u|^2}{u^2}+au^{p-1}+bu^{q-1}-\dfrac{u_t}{u}
\leq \overline{H_1}+H_2+\dfrac{n}{\beta}\dfrac{1}{t},\]
where
\[
\overline{H_1}=\dfrac{n}{\beta}\big(\sqrt{c_3}K_2+ |a|(1-p)M_1^{(p-1)} +|b|(1-q)M_1^{(q-1)} \big)
\]
\end{cor}
As an application, we get the following Harnack inequality.
\begin{prop}\label{mtheorem}
Let $(M, g(0))$ be a complete noncompact Riemannian manifold without boundary, and suppose that $g(t)$ evolve by $\dfrac{\partial g_{ij}}{\partial t}=2s_{ij}$ for $t\in [0, T ]$ and satisfy 
\[
Ric\geq −K_1 g,\quad −K_2 g\leq s \leq K_3g,\quad |\nabla s| \leq K_4.
\]
Also, assume that $\triangle a \leq \theta_a$, $\triangle b \leq \theta_b$, $|\nabla a| \leq \gamma_a$ and $|\nabla b| \leq \gamma_b$ in $M \times [0, T )$ for some constants $\theta_a ,\theta_b , \gamma_a$ and $\gamma_b$. Let $u(x,t)$ be a positive solution of (\ref{equ1}) in $M\times [0,T)$ where, $a$ and $b$ are positive constants. Then  for any constant $0 < \beta < 1$, if $\beta < p,q < 1$, for any points $(x_1,t_1)$ and $(x_2,t_2)$ on $M\times [0,T)$ with $0<t_1<t_2$, we have the following Harnack inequality,
\[
u(x_1,t_1)\leq u(x_2,t_2)(\dfrac{t_2}{t_1} )^{n\over \beta}e^{\Psi (x_1,x_2,t_1,t_2)+(\overline{H_1}+H_2)(t_2-t_1)}
\]
where $\Psi (x_1,x_2,t_1,t_2)=\inf_{\gamma}\int_{t_1}^{t_2} \dfrac{1}{4\beta}|\gamma'|^2 dt$, and $\overline{H_1}=\dfrac{n}{\beta}\big(\sqrt{c_3}K_2+ a(1-p)M_1^{(p-1)} +b(1-q)M_1^{(q-1)} \big)$, and
\begin{align*}
H_2 &=\Big[\dfrac{n^2 }{4\beta^2 (1-\beta)^2}\big( 2(1-\beta )K_3 +2\beta K_1 +\dfrac{3}{2}K_4 \big)^2\\
&+\dfrac{n}{\beta}\{M_1^{(p-1)}\theta_a + M_1^{(q-1)}\theta_b
+n \big( \dfrac{1}{\beta}(K_2 +K_3)^2 +\dfrac{3}{2}K_4\big)\}\\
&-\dfrac{n}{\beta}\{ \dfrac{[
 (p-\beta)M_1^{(p-1)}\gamma_a +(q-\beta)M_1^{(q-1)}\gamma_b]^2}{a(p-\beta)(p-1)M_1^{(p-1)}+b(q-\beta)(q-1)M_1^{(q-1)}}\} \Big]^{1\over 2}
\end{align*}
\end{prop}
%%%%%%%%%%%%%%%%%%%%%%%%%%%%%%%%%%%%%%%%%%%%%%%%%%%%%%%%%
\section{Methods and Proofs}

Let $u$ be a positive solution to (\ref{equ1}). Let $w=\ln u$, then a simple computation shows that $w$ satisfies the following equation
\begin{equation}\label{eq2}
w_t= \triangle w+|\nabla w|^2+ae^{(p-1)w}+be^{(q-1)w}
\end{equation}

Let $(M,g(t))$ be a smooth 1-parameter family of complete Riemannian metrics on a manifold $M$ evolving by equation
\begin{equation}\label{var}
\dfrac{\partial g_{ij}}{\partial t}=2s_{ij}
\end{equation}
for $t$ in some time interval $[0,T]$. We need the following lemmas to prove our main theorem.
\begin{lem}(\cite{lem})
If the metric evolves by (\ref{var}) then for any smooth function $w$, we have
\[
\dfrac{\partial}{\partial t}|\nabla w|^2= -2s(\nabla w,\nabla w)+2\nabla w \nabla w_t
\]
and,
\[ \dfrac{\partial}{\partial t}\triangle w =\triangle w_t -2s \nabla^2 w -2\nabla w \big( \mathrm{div} s-\dfrac{1}{2}\nabla (tr_g s)\big)
\]
where, $\mathrm{div} s$ denotes the divergence of $s$.
\end{lem}
\begin{lem}\label{lem}
Assume that $(M, g(t))$ satisfies the hypotheses of Proposition (\ref{maintheorem}). Then for any constant $0 < \beta < 1$ and $(x, t) \in Q_{R,T}$,if $\beta < p,q < 1$, we have
\begin{align*}
(\triangle -\frac{\partial}{\partial t})F\\
&\geq -2\nabla w\nabla F+t\Big\{\dfrac{\beta}{n}(w_t-|\nabla w |^2-ae^{(p-1)w} -be^{(q-1)w})^{2}\\
&+[a(p-\beta)(p-1)e^{(p-1)w}+b(q-\beta)(q-1)e^{(q-1)w}+2(\beta -1)K_3\\
 &-2\beta K_1 -\dfrac{3}{2}K_4]|\nabla w|^2 
+2(p-\beta)e^{(p-1)w}\nabla w\nabla a+2(q-\beta)e^{(q-1)w}\nabla w\nabla b\\
&+e^{(p-1)w}\triangle a+e^{(q-1)w}\triangle b
-n\big( \dfrac{1}{\beta}(K_2 +K_3)^2 +\dfrac{3}{2}K_4\big)
\Big\}\\
&-a(p-1)e^{(p-1)w}F-b(q-1)e^{(q-1)w}F-\frac{F}{t}
\end{align*}
where
\[F=t\big(\beta |\nabla w|^2+ae^{(p-1)w}+be^{(q-1)w}-w_t \big)
\]
\end{lem}
\begin{proof}
Define
\[F=t\big(\beta |\nabla w|^2+ae^{(p-1)w}+be^{(q-1)w}-w_t \big)
\]
By the Bochner formula, we can write
\[\triangle |\nabla w|^2\geq 2|\nabla^2 w|^2 + 2\nabla w \nabla(\triangle w)-2K_1 |\nabla w|^2.
\]
Note that
\begin{align*}
\triangle w_t =&(\triangle w)_t+2s \nabla^2 w +2\nabla w \big( \mathrm{div} s-\dfrac{1}{2}\nabla (tr_g s)\big)\\
=&w_{tt}-(|\nabla w|^2)_t-a_t e^{(p-1)w}-ae^{(p-1)w}-b_t e^{(q-1)w}-be^{(q-1)w}\\
&+2s \nabla^2 w +2\nabla w \big( \mathrm{div} s-\dfrac{1}{2}\nabla (tr_g s)\big)\\
=&2s(\nabla w,\nabla w)-2\nabla w\nabla w_t -a_t e^{(p-1)w}-ae^{(p-1)w}-b_t e^{(q-1)w}-be^{(q-1)w}\\
&+w_{tt}+2s \nabla^2 w +2\nabla w \big( \mathrm{div} s-\dfrac{1}{2}\nabla (tr_g s)\big)\\
\end{align*}
and
\begin{align*}
\triangle w &=-|\nabla w |^2-ae^{(p-1)w} -be^{(q-1)w} +w_t\\
&=(\dfrac{1}{\beta}-1)\big( ae^{(p-1)w} +be^{(q-1)w} -w_t\big)-\dfrac{F}{t\beta}\\
&=(\beta -1)|\nabla w |^2-\dfrac{F}{t}
\end{align*}
we can write,
\[\triangle F=t\Big(\beta \triangle |\nabla w|^2+\triangle(ae^{(p-1)w})+\triangle(be^{(q-1)w})-\triangle w_t  \Big) \]
According to the above computations, we obtain
\begin{align*}
\beta \triangle |\nabla w|^2 \geq &2\beta |\nabla^2 w|^2 + 2\beta \nabla w \nabla(\triangle w)-2\beta K_1 |\nabla w|^2\\
=&2\beta |\nabla^2 w|^2 + 2\beta \nabla w \nabla \Big([(\dfrac{1}{\beta}-1)\big( ae^{(p-1)w} +be^{(q-1)w} -w_t\big)-\dfrac{F}{t\beta}]\Big)\\-&2\beta K_1 |\nabla w|^2\\
=&2\beta|\nabla^{2}w|^{2}-\frac{2}{t}\nabla w\nabla F+2(1-\beta)e^{(p-1)w}\nabla w\nabla a
+2(1-\beta)e^{(q-1)w}\nabla w\nabla b\\
&+2a(1-\beta)(p-1)e^{(p-1)w}|\nabla w|^{2}+2b(1-\beta)(q-1)e^{(q-1)w}|\nabla w|^{2}\\
&+2(1-\beta)\nabla w\nabla w_{t}-2K_{1}\beta |\nabla w|^{2}
\end{align*}
and, we know
\begin{eqnarray*}
\triangle(ae^{(p-1)w})&=&e^{(p-1)w}\triangle a+2(p-1)e^{(p-1)w}\nabla w\nabla a+a(p-1)^{2}e^{(p-1)w}|\nabla w|^{2}\\&&+a(p-1)e^{(p-1)w}\triangle w\\
&=&e^{(p-1)w}\triangle a+2(p-1)e^{(p-1)w}\nabla w\nabla a+a(p-1)^{2}e^{(p-1)w}|\nabla w|^{2}\\&&+a(p-1)e^{(p-1)w}[(\beta-1)|\nabla w|^{2}-\frac{F}{t}]
\end{eqnarray*}
so we have
\begin{eqnarray*}
\triangle F&\geq& t\Big\{2\beta|\nabla^{2}w|^{2}-\frac{2}{t}\nabla w\nabla F+2(1-\beta)e^{(p-1)w}\nabla w\nabla a
+2(1-\beta)e^{(q-1)w}\nabla w\nabla b\\
&&+2a(1-\beta)(p-1)e^{(p-1)w}|\nabla w|^{2}+2b(1-\beta)(q-1)e^{(q-1)w}|\nabla w|^{2}\\&&
+2(1-\beta)\nabla w\nabla w_{t}-2k_{1}\beta |\nabla w|^{2}
+e^{(p-1)w}\triangle a+2(p-1)e^{(p-1)w}\nabla w\nabla a\\&&+a(p-1)^{2}e^{(p-1)w}|\nabla w|^{2}+a(p-1)e^{(p-1)w}[(\beta-1)|\nabla w|^{2}-\frac{F}{t}]\\&&
+e^{(q-1)w}\triangle b+2(q-1)e^{(q-1)w}\nabla w\nabla b+b(q-1)^{2}e^{(q-1)w}|\nabla w|^{2}\\&&+b(q-1)e^{(q-1)w}[(\beta-1)|\nabla w|^{2}-\frac{F}{t}]
-\Big[ w_{tt}-(|\nabla w|^{2})_{t}-a_{t}e^{(p-1)w}\\&&-a(p-1)e^{(p-1)w}w_{t}-b_{t}e^{(q-1)w}-b(q-1)e^{(q-1)w}w_{t}+2s\nabla^{2}w\\&&+2\nabla w(\mathrm{div} s-\dfrac{1}{2}\nabla (tr_g s))\Big]\Big\}
\end{eqnarray*}
and
\begin{eqnarray*}
F_{t}&=& \frac{F}{t}+t\Big\{2\beta (|\nabla w|^{2})_t+a_{t}e^{(p-1)w}+a(p-1)e^{(p-1)w}w_{t}+b_{t}e^{(q-1)w}\\
&+&b(q-1)e^{(q-1)w}w_{t}-w_{tt}\Big\}\\
&=&\frac{F}{t}+t\Big\{2\beta\nabla w\nabla w_{t}-2\beta s(\nabla w,\nabla w)+a_{t}e^{(p-1)w}+a(p-1)e^{(p-1)w}w_{t}+b_{t}e^{(q-1)w}\\&&+b(q-1)e^{(q-1)w}w_{t}-w_{tt}\Big\}
\end{eqnarray*}
This equation implies that
\begin{eqnarray*}
\Box F&=&(\triangle -\frac{\partial}{\partial t})F\geq -2\nabla w\nabla F+t\Big\{2\beta|\nabla^{2}w|^{2}+2(\beta-1) s(\nabla w,\nabla w)\\&&+a(p-\beta)(p-1)e^{(p-1)w}|\nabla w|^{2}+b(q-\beta)(q-1)e^{(q-1)w}|\nabla w|^{2}\\&&
+2(p-\beta)e^{(p-1)w}\nabla w\nabla a+2(q-\beta)e^{(q-1)w}\nabla w\nabla b+e^{(p-1)w}\triangle a+e^{(q-1)w}\triangle b\\&&-2K_{1}\beta |\nabla w|^{2}-2s\nabla^{2}w-2\nabla w(\mathrm{div} s-\dfrac{1}{2}\nabla (tr_g s))\Big\}-a(p-1)e^{(p-1)w}F\\&&-b(q-1)e^{(q-1)w}F-\frac{F}{t}
\end{eqnarray*}
By our assumptions, we have
\[
-(K_{2}+K_{3})g\leq s\leq(K_{2}+K_{3})g\]
which implies that
\[
|s|^{2}\leq (K_{2}+K_{3})^{2}|g|^{2}=n(K_{2}+K_{3})^{2}
\]
Using Young's inequality and applying those bounds yields
\begin{equation*}
|s\nabla^{2} w|\leq \frac{\beta}{2}|\nabla^{2}w|^{2}+\frac{1}{2\beta}|s|^{2}\leq  \frac{\beta}{2}|\nabla^{2}w|^{2}+\frac{n}{2\beta}(K_{2}+K_{3})^{2}
\end{equation*}
on the other hand,
\begin{equation*}
|\mathrm{div} s-\dfrac{1}{2}\nabla (tr_g s)|=|g^{ij}\nabla_{i}s_{jl}-\frac{1}{2}g^{ij}\nabla_{l}s_{ij}|\leq \frac{3}{2}|g||\nabla s|\leq \frac{3}{2}\sqrt{n}K_{4}
\end{equation*}
Finally, with the help of the following inequality, 
\begin{equation*}
|\nabla^{2}w|^{2}\geq \frac{1}{n}(tr \nabla^{2} w)^{2}=\frac{1}{n}(\triangle w)^{2}=\frac{1}{n}(-|\nabla w |^2-ae^{(p-1)w} -be^{(q-1)w} +w_t)^{2}
\end{equation*}
we obtain
\begin{align*}
(\triangle -\frac{\partial}{\partial t})F\\
&\geq -2\nabla w\nabla F+t\Big\{\dfrac{\beta}{n}(w_t-|\nabla w |^2-ae^{(p-1)w} -be^{(q-1)w})^{2}\\
&+a(p-\beta)(p-1)e^{(p-1)w}|\nabla w|^{2}+b(q-\beta)(q-1)e^{(q-1)w}|\nabla w|^{2}\\
&+2(p-\beta)e^{(p-1)w}\nabla w\nabla a+2(q-\beta)e^{(q-1)w}\nabla w\nabla b+e^{(p-1)w}\triangle a+e^{(q-1)w}\triangle b\\
&+2(\beta -1)K_3|\nabla w|^2 -2\beta K_1|\nabla w|^2 -\dfrac{n}{\beta}(K_2 +K_3)^2\\
&-3\sqrt{n}K_4 |\nabla w | \Big\}-a(p-1)e^{(p-1)w}F-b(q-1)e^{(q-1)w}F-\frac{F}{t}
\end{align*}
Applying AM-GM inequality, we can write
\begin{equation*}
3\sqrt{n}K_4 |\nabla w | \leq 3K_4(\dfrac{n}{2}+\dfrac{|\nabla w|^2}{2})       
\end{equation*}
we get
\begin{align*}
(\triangle -\frac{\partial}{\partial t})F\\
&\geq -2\nabla w\nabla F+t\Big\{\dfrac{\beta}{n}(w_t-|\nabla w |^2-ae^{(p-1)w} -be^{(q-1)w})^{2}\\
&+[a(p-\beta)(p-1)e^{(p-1)w}+b(q-\beta)(q-1)e^{(q-1)w}+2(\beta -1)K_3\\
 &-2\beta K_1 -\dfrac{3}{2}K_4]|\nabla w|^2 
+2(p-\beta)e^{(p-1)w}\nabla w\nabla a+2(q-\beta)e^{(q-1)w}\nabla w\nabla b\\
&+e^{(p-1)w}\triangle a+e^{(q-1)w}\triangle b
-n\big( \dfrac{1}{\beta}(K_2 +K_3)^2 +\dfrac{3}{2}K_4\big)
\Big\}\\
&-a(p-1)e^{(p-1)w}F-b(q-1)e^{(q-1)w}F-\frac{F}{t}
\end{align*}
This completes the proof.
\end{proof}
Let's take a cut-off function $\tilde{\varphi}$ defined on $[0,\infty )$ such that $0\leq \tilde{\varphi} (r)\leq 1$, $\tilde{\varphi} (r)=1$ for $r\in [0,1]$ and, $\tilde{\varphi} (r)=0$ for $r\in [2,\infty) $. Furthermore $\tilde{\varphi}$ satisfies the following inequalities for some positive constants $c_1$ and $c_2$.
\[-\dfrac{\tilde{\varphi}'(r)}{\tilde{\varphi}^{1\over 2}(r)}\leq c_1 ,\qquad \tilde{\varphi}''(r)\geq -c_2\]
Define $r(x,t):=d(x,x_0,t)$ and, set
\[\varphi (x,t)=\tilde{\varphi} (\dfrac{r(x,t)}{R})\]
Using a similar argument of Calabi \cite{Calabi}, we can assume $\varphi (x,t)\in C^2 (M)$ with support in $Q_{2R,T}$. A direct calculation indicates that on $Q_{2R,T}$, we have
\begin{equation}\label{eq1}
\dfrac{|\nabla \varphi |^2}{\varphi}\leq \dfrac{c_1^2}{R^2}
\end{equation}
According to the Laplace comparison theorem in \cite{1}, we can write
\begin{equation}\label{eq2}
\triangle \varphi \geq -\dfrac{(n-1)(1+\sqrt{K_1}R)c_1^2+c^2}{R^2}
\end{equation}
For any $0<T_1<T$, suppose that $\varphi F$ attains it maximum value at the point $(x_0,t_0)$ in the cube $Q_{2R,T_1}$. We can assume that this maximum value is positive (otherwise the proof of our main theorem will be trivial). At the maximum point $(x_0,t_0)$, we have
\[\nabla (\varphi F)=0,\quad \triangle (\varphi F)\leq 0, \quad (\varphi F)_t\geq 0, \]
which follows that
\[0\geq (\triangle -\dfrac{\partial}{\partial t})(\varphi F)=(\triangle \varphi)F-\varphi_t F+\varphi (\triangle -\dfrac{\partial}{\partial t})F+2\nabla \varphi \nabla F
\]
So, we can write
\begin{equation}\label{eq3}
(\triangle \varphi)F-\varphi_t F+\varphi (\triangle -\dfrac{\partial}{\partial t})F-2F\varphi^{-1} |\nabla \varphi |^2\leq 0.
\end{equation}
Also, we know (see \cite{lem}, p. 494) there exists a positive constant $c_3$ such that
$$ -\varphi_t F\geq -\sqrt{c_3}K_2F.$$
The inequality (\ref{eq3}) together with the inequalities (\ref{eq1}) and (\ref{eq2}) yield
\begin{equation}\label{eq4}
\varphi (\triangle-\dfrac{\partial}{\partial t})F\leq HF,\end{equation}
where 
$$H=\dfrac{(n-1)(1+\sqrt{K_1}R)c_1^2+c_2+2c_1^2}{R^2}+\sqrt{c_3}K_2$$
\textbf{Proof of Proposition (\ref{maintheorem}):}
At the maximum point $(x_0,t_0)$, by (\ref{eq4}) and Lemma \ref{lem}, we have
\begin{align*}
&0\geq \varphi (\triangle -\dfrac{\partial}{\partial t})F-HF \geq -HF +\varphi \{ -2\nabla w\nabla F+\dfrac{\beta t_0}{n}(w_t-|\nabla w |^2-ae^{(p-1)w} \\
&-be^{(q-1)w})^{2}+t_0[a(p-\beta)(p-1)e^{(p-1)w}+b(q-\beta)(q-1)e^{(q-1)w}+2(\beta -1)K_3\\
 &-2\beta K_1 -\dfrac{3}{2}K_4]|\nabla w|^2 
+2t_0 (p-\beta)e^{(p-1)w}\nabla w\nabla a+2t_0 (q-\beta)e^{(q-1)w}\nabla w\nabla b\\
&+t_0 e^{(p-1)w}\triangle a+t_0 e^{(q-1)w}\triangle b
-nt_0\big( \dfrac{1}{\beta}(K_2 +K_3)^2 +\dfrac{3}{2}K_4\big)\\
&-a(p-1)e^{(p-1)w}F-b(q-1)e^{(q-1)w}F-\frac{F}{t_0} \}\\
\geq & -HF +2F\nabla w\nabla \varphi +\dfrac{\beta t_0}{n}\varphi (w_t-|\nabla w |^2-ae^{(p-1)w} -be^{(q-1)w})^{2}\\
&+t_0  \varphi [a(p-\beta)(p-1)e^{(p-1)w}+b(q-\beta)(q-1)e^{(q-1)w}+2(\beta -1)K_3\\
 &-2\beta K_1 -\dfrac{3}{2}K_4]|\nabla w|^2 
+2t_0 \varphi (p-\beta)e^{(p-1)w}\nabla w\nabla a+2t_0 \varphi (q-\beta)e^{(q-1)w}\nabla w\nabla b\\
&+t_0 \varphi e^{(p-1)w}\triangle a+t_0 \varphi e^{(q-1)w}\triangle b
-nt_0 \varphi \big( \dfrac{1}{\beta}(K_2 +K_3)^2 +\dfrac{3}{2}K_4\big)\\
&-a(p-1)e^{(p-1)w}\varphi F-b(q-1)e^{(q-1)w}\varphi F-\varphi t_0^{-1}F\\
\geq & -HF +2F\nabla w\nabla \varphi +\dfrac{\beta t_0}{n}\varphi (w_t-|\nabla w |^2-ae^{(p-1)w} -be^{(q-1)w})^{2}\\
&-t_0  \varphi [|a|(p-\beta)(p-1)M_1^{(p-1)}+|b|(q-\beta)(q-1)M_1^{(q-1)}+2(1-\beta )K_3\\
 &+2\beta K_1 +\dfrac{3}{2}K_4]|\nabla w|^2
 +2t_0 \varphi (\beta-p)M_1^{(p-1)}\gamma_a |\nabla w|+2t_0 \varphi (\beta-q)M_1^{(q-1)}\gamma_b |\nabla w|\\
 &-t_0 \varphi M_1^{(p-1)}\theta_a -t_0 \varphi M_1^{(q-1)}\theta_b
-nt_0 \varphi \big( \dfrac{1}{\beta}(K_2 +K_3)^2 +\dfrac{3}{2}K_4\big)\\
&+|a|(p-1)M_1^{(p-1)}\varphi F+|b|(q-1)M_1^{(q-1)}\varphi F-\varphi t_0^{-1}F\\
\geq &  -HF +2F\nabla w\nabla \varphi +\dfrac{\beta t_0}{n}\varphi (w_t-|\nabla w |^2-ae^{(p-1)w} -be^{(q-1)w})^{2}\\
&-t_0  \varphi [|a|(p-\beta)(p-1)M_1^{(p-1)}+|b|(q-\beta)(q-1)M_1^{(q-1)}]|\nabla w|^2 \\
&-t_0  \varphi [2(1-\beta )K_3
 +2\beta K_1 +\dfrac{3}{2}K_4]|\nabla w|^2 -t_0\varphi [
 2(p-\beta)M_1^{(p-1)}\gamma_a \\&+2(q-\beta)M_1^{(q-1)}\gamma_b]|\nabla w|\\
 &-t_0 \varphi [M_1^{(p-1)}\theta_a + M_1^{(q-1)}\theta_b
+n \big( \dfrac{1}{\beta}(K_2 +K_3)^2 +\dfrac{3}{2}K_4\big)]\\
&+|a|(p-1)M_1^{(p-1)}\varphi F+|b|(q-1)M_1^{(q-1)}\varphi F-\varphi t_0^{-1}F
\end{align*}
For the sake of simplicity, set
\begin{align*}
\widetilde{C_1}&=2(1-\beta )K_3 +2\beta K_1 +\dfrac{3}{2}K_4
\\
\widetilde{C_2}&=M_1^{(p-1)}\theta_a + M_1^{(q-1)}\theta_b
+n \big( \dfrac{1}{\beta}(K_2 +K_3)^2 +\dfrac{3}{2}K_4\big)
\end{align*}
and
\[\widetilde{C_3}=-\dfrac{[
 (p-\beta)M_1^{(p-1)}\gamma_a +(q-\beta)M_1^{(q-1)}\gamma_b]^2}{|a|(p-\beta)(p-1)M_1^{(p-1)}+|b|(q-\beta)(q-1)M_1^{(q-1)}}
\]
Using the inequality $ax^2+bx\leq -\dfrac{b^2}{4a}$ which holds for $a<0$, we obtain
\begin{align*}
0\geq & -HF +2F\nabla w\nabla \varphi +\dfrac{\beta t_0}{n}\varphi (w_t-|\nabla w |^2-ae^{(p-1)w} -be^{(q-1)w})^{2}\\
&-t_0 \varphi [\widetilde{C_3}+\widetilde{C_2}+\widetilde{C_1}|\nabla w|^2]
+|a|(p-1)M_1^{(p-1)}\varphi F+|b|(q-1)M_1^{(q-1)}\varphi F-\varphi t_0^{-1}F
\end{align*}
Noting the fact that $0<\varphi <1$ and multiplying both sides of the above inequality by $t_0 \varphi$, leads to
\begin{align*}
0 \geq &-Ht_0\varphi F+2t_0\varphi F\nabla w\nabla \varphi +\dfrac{\beta t_0^2}{n}\varphi^2 (w_t-|\nabla w |^2-ae^{(p-1)w} -be^{(q-1)w})^{2}\\
-\widetilde{C_1}t_0^2&\varphi^2 |\nabla w|^2 -(\widetilde{C_2}+\widetilde{C_3})t_0^2\varphi^2 +|a|(p-1)M_1^{(p-1)}t_0\varphi F+|b|(q-1)M_1^{(q-1)}t_0\varphi F-\varphi F\\
\geq -Ht_0&\varphi F -\dfrac{2c_1}{R}t_0\varphi F|\nabla w|\varphi^{3\over 2} +|a|(p-1)M_1^{(p-1)}t_0\varphi F+|b|(q-1)M_1^{(q-1)}t_0\varphi F-\varphi F\\
& +\dfrac{\beta t_0^2}{n}\varphi^2 [(w_t-|\nabla w |^2-ae^{(p-1)w} -be^{(q-1)w})^{2} -\dfrac{n}{\beta}\widetilde{C_1} |\nabla w|^2 ]-(\widetilde{C_2}+\widetilde{C_3})t_0^2\varphi^2 
\end{align*}
where in the last inequality the following fact is applied 
\[-2\varphi \nabla w \nabla F=2F\nabla w \nabla \varphi \geq -2F|\nabla w||\nabla\varphi|\geq -\dfrac{2c_1}{R}\varphi^{1\over 2}F|\nabla w| \]
Assume that
\[ y=\varphi |\nabla w|^2, \qquad z=\varphi (ae^{(p-1)w} +be^{(q-1)w}-w_t) \]
So, we can write
\begin{align*}
0\geq & \varphi F (-Ht_0+|a|(p-1)M_1^{(p-1)}t_0+|b|(q-1)M_1^{(q-1)}t_0 -1 )-\dfrac{2c_1}{R}t_0 F|\nabla w|\varphi^{3\over 2}\\
& +\dfrac{\beta t_0^2}{n}\varphi^2 [(w_t-|\nabla w |^2-ae^{(p-1)w} -be^{(q-1)w})^{2} -\dfrac{n}{\beta}\widetilde{C_1} |\nabla w|^2 ]-(\widetilde{C_2}+\widetilde{C_3})t_0^2\varphi^2 \\
\geq & \varphi F (-Ht_0+|a|(p-1)M_1^{(p-1)}t_0+|b|(q-1)M_1^{(q-1)}t_0 -1 )\\
&+\dfrac{\beta t_0^2}{n}\{ (y-z)^2-\dfrac{n}{\beta}\widetilde{C_1}y-2nc_1 R^{-1} y^{1\over 2}(y-\dfrac{1}{\beta}z )\}-(\widetilde{C_2}+\widetilde{C_3})t_0^2.
\end{align*}
For all $a,b>0$ the inequality $ax^2-bx\geq -\dfrac{b^2}{4a}$ holds for every real number $x$. Using this inequality, we obtain
\begin{align*}
&\dfrac{\beta t_0^2}{n}\{ (y-z)^2-\dfrac{n}{\beta}\widetilde{C_1}y-2nc_1 R^{-1} y^{1\over 2}(y-\dfrac{1}{\beta}z )\}\\
=&\dfrac{\beta t_0^2}{n}\{ \beta^2 (y-\dfrac{z}{\beta} )^2 +(1-\beta^2)y^2 -\dfrac{n}{\beta}\widetilde{C_1}y+[2(\beta-\beta^2) y-2\dfrac{nc_1}{R} y^{1\over 2}](y-\dfrac{z}{\beta})\}\\
\geq & \dfrac{\beta t_0^2}{n} \{ \beta^2 (y-\dfrac{z}{\beta} )^2  -\dfrac{n^2 \widetilde{C_1}^2}{4\beta^2 (1-\beta)^2}-\dfrac{n^2c_1^2}{2R^2(\beta -\beta^2)} (y-\dfrac{z}{\beta}) \}\\
=& \dfrac{\beta}{n}(\varphi F)^2  -\dfrac{n \widetilde{C_1}^2 t_0^2}{4\beta (1-\beta)^2}-\dfrac{n c_1^2 t_0}{2R^2(\beta -\beta^2)} (\varphi F).
\end{align*}
Hence,
\begin{align*}
&\dfrac{\beta}{n}(\varphi F)^2+\big[ -Ht_0+|a|(p-1)M_1^{(p-1)}t_0+|b|(q-1)M_1^{(q-1)}t_0 -1 -\dfrac{n c_1^2 t_0}{2R^2(\beta -\beta^2)}\big](\phi F)\\&- \dfrac{n \widetilde{C_1}^2 t_0^2}{4\beta (1-\beta)^2}-(\widetilde{C_2}+\widetilde{C_3})t_0^2 \leq 0.
\end{align*}
As we know, the inequality $Ax^2-2Bx\leq C$, yields $x\leq \dfrac{2B}{A}+\sqrt{\dfrac{C}{A} }$. So, we get
\begin{align*}
\varphi F \leq \dfrac{n}{\beta} &\Big( Ht_0+|a|(1-p)M_1^{(p-1)}t_0+|b|(1-q)M_1^{(q-1)}t_0 +1 +\dfrac{n c_1^2 t_0}{2R^2(\beta -\beta^2)}\Big)\\
&+\Big[\dfrac{n}{\beta}( \dfrac{n \widetilde{C_1}^2 }{4\beta (1-\beta)^2}+\widetilde{C_2}+\widetilde{C_3}) \Big]^{1\over 2}t_0.
\end{align*}
If $d(x,x_0,T_1),R$, we know that $\varphi (x,T_1)=1$. Then
\begin{align*}
F(x,T_1)&=T_1 (\beta |\nabla w|^2+ae^{(p-1)w}+be^{(q-1)w}-w_t )\\
&\leq \varphi F(x_0 ,t_0)\\
&\leq \dfrac{n}{\beta} \Big( Ht_0+|a|(1-p)M_1^{(p-1)}t_0+|b|(1-q)M_1^{(q-1)}t_0 +1 +\dfrac{n c_1^2 t_0}{2R^2(\beta -\beta^2)}\Big)\\
&+\Big[\dfrac{n}{\beta}( \dfrac{n \widetilde{C_1}^2 }{4\beta (1-\beta)^2}+\widetilde{C_2}+\widetilde{C_3}) \Big]^{1\over 2}t_0.
\end{align*}
Since $T_1$ was supposed to be arbitrary, we can get the assertion.
\\
\textbf{Proof of Proposition (\ref{mtheorem}):}
For any points $(x_1,t_1)$ and $(x_2,t_2)$ on $M\times [0,T)$ with $0<t_1<t_2$, we take a curve $\gamma(t)$ parametrized with $\gamma (t_1)=x_1$ and $\gamma(t_2)=x_2$. In the ray of Corollary (\ref{corol}), one can get
\begin{align*}
&\log u(x_2,t_2)-\log u(x_1,y_1)\\
=& \int_{t_1}^{t_2} \big( (\log u)_t +\langle \nabla \log u ,\gamma' \rangle \big) dt\\
\geq & \int_{t_1}^{t_2} \Big( \beta |\nabla \log u|^2+au^{p-1}+bu^{q-1}-\overline{H_1}-H_2-\dfrac{n}{\beta t} -|\nabla \log u||\gamma'| \Big)dt\\
\geq & -\int_{t_1}^{t_2} \big( \dfrac{1}{4\beta}|\gamma'|^2 -au^{p-1}-bu^{q-1}+\overline{H_1}+H_2+\dfrac{n}{\beta t}\Big)dt\\
\geq & -\Big(\log(\dfrac{t_2}{t_1} )^{n\over \beta}+ (\overline{H_1}+H_2)(t_2-t_1)+\int_{t_1}^{t_2} \dfrac{1}{4\beta}|\gamma'|^2 dt\Big)
\end{align*}
which means
\[
\log \dfrac{u(x_1,t_1)}{u(x_2,t_2)}\leq \log(\dfrac{t_2}{t_1} )^{n\over \beta}+ (\overline{H_1}+H_2)(t_2-t_1)+\int_{t_1}^{t_2} \dfrac{1}{4\beta}|\gamma'|^2 dt
\]
Hence,
\[
u(x_1,t_1)\leq u(x_2,t_2)(\dfrac{t_2}{t_1} )^{n\over \beta}e^{\Psi (x_1,x_2,t_1,t_2)+(\overline{H_1}+H_2)(t_2-t_1)}
\]
where $\Psi (x_1,x_2,t_1,t_2)=\inf_{\gamma}\int_{t_1}^{t_2} \dfrac{1}{4\beta}|\gamma'|^2 dt$, and $\overline{H_1}=\dfrac{n}{\beta}\big(\sqrt{c_3}K_2+ a(1-p)M_1^{(p-1)} +b(1-q)M_1^{(q-1)} \big)$, and
\begin{align*}
H_2 &=\Big[\dfrac{n^2 }{4\beta^2 (1-\beta)^2}\big( 2(1-\beta )K_3 +2\beta K_1 +\dfrac{3}{2}K_4 \big)^2\\
&+\dfrac{n}{\beta}\{M_1^{(p-1)}\theta_a + M_1^{(q-1)}\theta_b
+n \big( \dfrac{1}{\beta}(K_2 +K_3)^2 +\dfrac{3}{2}K_4\big)\}\\
&-\dfrac{n}{\beta}\{ \dfrac{[
 (p-\beta)M_1^{(p-1)}\gamma_a +(q-\beta)M_1^{(q-1)}\gamma_b]^2}{a(p-\beta)(p-1)M_1^{(p-1)}+b(q-\beta)(q-1)M_1^{(q-1)}}\} \Big]^{1\over 2}
\end{align*}

% ------------------------------------------------------------------------
\end{document}